\documentclass{article}
\usepackage{amsfonts}
\usepackage{amsmath}
\usepackage{harvard}

\newtheorem{theorem}{Theorem}

\newtheorem{corollary}[theorem]{Corollary}

\newtheorem{example}[theorem]{Example}

\newenvironment{proof}[1][Proof]{\noindent\textbf{#1.} }{\ \rule{0.5em}{0.5em}}

\input{tcilatex}

\begin{document}

\title{Integer powers of certain complex tridiagonal matrices and some
complex factorizations}
\author{Durmu\c{s} Bozkurt\thanks{%
dbozkurt@selcuk.edu.tr} \& \c{S}. Burcu Bozkurt\thanks{%
sbbozkurt@selcuk.edu.tr} \\
%EndAName
Department of Mathematics, Science Faculty of Sel\c{c}uk University}
\maketitle

\begin{abstract}
In this paper, we obtain a general expression for the entries of the $r$th ($%
r\in \mathbb{Z}$) power of a certain $n\times n$ complex tridiagonal matrix.
In addition, we get the complex factorizations of Fibonacci polynomials,
Fibonacci and Pell numbers.
\end{abstract}

\section{Introduction}

\bigskip In order to solve some difference equations, differential and delay
differential equations and boundary value problems, we need to compute the
arbitrary integer powers of a square matrix.

The integer powers of an $n\times n$ matrix $A$ is computed using the
well-known expression $A^{r}=PJ^{r}P^{-1}$ [?], where $J$ \ is the Jordan's
form and $P$ is the eigenvector matrix of $A$, respectively.

Recently, the calculations of integer powers and eigenvalues of tridiagonal
matrices have been well studied. For instance, Rimas [1-4] obtained the
positive integer powers of certain tridiagonal matrices of odd and even
order. \"{O}tele\c{s} and Akbulak [6,7] generalized the results obtained in
[1-4]. Guti\'{e}rrez [8,10] calculated the powers of tridiagonal matrices
with constant diagonal. For detailed information on the powers and the
eigenvalues of tridiagonal matrices, we may refer to the reader [5,9].

\bigskip In [12], Cahill et al. considered the following tridagonal matrix 
\begin{equation*}
H(n)=\left( 
\begin{array}{ccccc}
h_{1,1} & h_{1,2} &  &  & 0 \\ 
h_{2,1} & h_{2,2} & h_{2,3} &  &  \\ 
& h_{3,2} & h_{3,3} & \ddots &  \\ 
&  & \ddots & \ddots & h_{n-1,n} \\ 
0 &  &  & h_{n,n-1} & h_{n,n}%
\end{array}%
\right)
\end{equation*}%
and using the succesive determinants they computed determinant of $H(n)$ as 
\begin{equation*}
\left\vert H(n)\right\vert =h_{n,n}\left\vert H(n-1)\right\vert
-h_{n-1,n}h_{n,n-1}\left\vert H(n-2)\right\vert
\end{equation*}%
with initial conditions $\left\vert H(1)\right\vert =h_{1,1},$ $\left\vert
H(2)\right\vert =h_{1,1}h_{2,2}-h_{1,2}h_{2,1}.$

Let $\{H^{\dagger }(n),n=1,2,\ldots \}$ be the sequence of tridiagonal
matrices as in the form%
\begin{equation*}
H^{\dagger }(n)=\left( 
\begin{array}{ccccc}
h_{1,1} & -h_{1,2} &  &  &  \\ 
-h_{2,1} & h_{2,2} & -h_{2,3} &  &  \\ 
& -h_{3,2} & h_{3,3} & \ddots &  \\ 
&  & \ddots & \ddots & -h_{n-1,n} \\ 
&  &  & -h_{n,n-1} & h_{n,n}%
\end{array}%
\right) .
\end{equation*}%
Then%
\begin{equation}
\det (H(n))=\det (H^{\dagger }(n)).  \tag{1}  \label{1}
\end{equation}%
Let $T$ and $T^{\dagger }$ be $n\times n$ tridiagonal matrices as the
following%
\begin{equation*}
T:=\left( 
\begin{array}{cccccc}
0 & 2 &  &  &  &  \\ 
1 & 0 & 1 &  &  &  \\ 
& 1 & 0 & 1 &  &  \\ 
&  & \ddots & \ddots & \ddots &  \\ 
&  &  & 1 & 0 & 1 \\ 
&  &  &  & 2 & 0%
\end{array}%
\right) \ [1],
\end{equation*}%
\begin{equation*}
T^{\dagger }:=\left( 
\begin{array}{cccccc}
1 & 1 &  &  &  &  \\ 
1 & 0 & 1 &  &  &  \\ 
& 1 & 0 & 1 &  &  \\ 
&  & \ddots & \ddots & \ddots &  \\ 
&  &  & 1 & 0 & 1 \\ 
&  &  &  & 1 & 1%
\end{array}%
\right) [2].
\end{equation*}%
From (1), it is clear that%
\begin{equation*}
\left\vert \lambda I_{n}-T\right\vert =\left\vert 
\begin{array}{cccccc}
\mu & -2 &  &  &  &  \\ 
-1 & \mu & -1 &  &  &  \\ 
& -1 & \mu & -1 &  &  \\ 
&  & \ddots & \ddots & \ddots &  \\ 
&  &  & -1 & \mu & -1 \\ 
&  &  &  & -2 & \mu%
\end{array}%
\right\vert
\end{equation*}%
and%
\begin{equation*}
\left\vert \lambda I_{n}-T^{\dagger }\right\vert =\left\vert 
\begin{array}{cccccc}
\mu ^{\dagger }-1 & -1 &  &  &  &  \\ 
-1 & \mu ^{\dagger } & -1 &  &  &  \\ 
& -1 & \mu ^{\dagger } & -1 &  &  \\ 
&  & \ddots & \ddots & \ddots &  \\ 
&  &  & -1 & \mu ^{\dagger } & -1 \\ 
&  &  &  & -1 & \mu ^{\dagger }-1%
\end{array}%
\right\vert .
\end{equation*}%
By [1, p. 3] and [2, p. 2], the eigenvalues of $T$ and $T^{\dagger }$ are
obtained as%
\begin{equation*}
\mu _{k}=2\cos \left( \frac{(k-1)\pi }{n-1}\right) ,\ k=\overline{1,n}\ [1,\
p.3]
\end{equation*}%
and%
\begin{equation*}
\mu _{k}^{\dagger }=-2\cos \left( \frac{k\pi }{n}\right) ,\ k=\overline{1,n}%
\ [2,\ p.2]
\end{equation*}%
respectively.

Let

\begin{equation}
A:=\left( 
\begin{array}{cccccc}
a & 2b &  &  &  & 0 \\ 
b & a & -b &  &  &  \\ 
& -b & a & -b &  &  \\ 
&  & \ddots & \ddots & \ddots &  \\ 
&  &  & -b & a & b \\ 
0 &  &  &  & 2b & a%
\end{array}%
\right)  \tag{2}  \label{2}
\end{equation}%
and%
\begin{equation}
A^{^{\dagger }}:=\left( 
\begin{array}{cccccc}
a+b & b &  &  &  & 0 \\ 
b & a & -b &  &  &  \\ 
& -b & a & -b &  &  \\ 
&  & \ddots & \ddots & \ddots &  \\ 
&  &  & -b & a & b \\ 
0 &  &  &  & b & a+b%
\end{array}%
\right)  \tag{3}  \label{3}
\end{equation}%
be the tridiagonal matrices with $a$ and $b\neq 0$ are complex numbers. In
this paper, we obtain the eigenvalues and eigenvectors of an $n\times n$
complex tridiagonal matrices in (2) and (3) and calculate the integer powers
of the matrix in (2) for $n$ is odd order.

\section{Eigenvalues and eigenvectors of $A$ and $A^{\dagger }$}

\begin{theorem}
\bigskip Let $A$ be an $n\times n$ tridiagonal matrix given by (2). Then the
eigenvalues and eigenvectors of A are 
\begin{equation}
\lambda _{k}=a+2b\cos \left( \frac{(k-1)\pi }{n-1}\right) ,\ k=\overline{1,n}
\tag{4}  \label{4}
\end{equation}%
and%
\begin{equation*}
x_{jk}=\left\{ 
\begin{array}{l}
T_{j-1}(m_{k}),\ \ \ \ \ \ \ \ \ \ j=1,2,n-1,n \\ 
(-1)^{j}T_{j-1}(m_{k}),\ j=\overline{3,n-2}%
\end{array}%
\right. ;k=\overline{1,n}
\end{equation*}%
where $m_{k}=\frac{\lambda _{k}-a}{2b},T_{s}(.)$ is the $s-$th degree
Chebyshev polynomial of the first kind [11, p. 14].
\end{theorem}

\begin{proof}
\bigskip Let $B$ be the following $n\times n$ tridiagonal matrix%
\begin{equation}
B:=\left( 
\begin{array}{cccccc}
c & 2 &  &  &  &  \\ 
1 & c & -1 &  &  &  \\ 
& -1 & c & -1 &  &  \\ 
&  & \ddots & \ddots & \ddots &  \\ 
&  &  & -1 & c & 1 \\ 
&  &  &  & 2 & c%
\end{array}%
\right)  \tag{5}  \label{5}
\end{equation}%
where $c=\frac{a}{b}.$ Then the characteristic polynomials of $B$ are%
\begin{equation}
p_{n}(t)=(t^{2}-4)\Delta _{n-2}(t)  \tag{6}  \label{6}
\end{equation}%
where $t=\lambda -c$ and 
\begin{equation}
\Delta _{n}(t)=t\Delta _{n-1}(t)-\Delta _{n-2}(t)  \tag{7}  \label{7}
\end{equation}%
with initial conditions $\Delta _{0}(t)=1,\Delta _{1}(t)=t$ and $\Delta
_{2}(t)=t^{2}-1.$ Note that the solution of the difference equation in (7)
is $\Delta _{n}(t)=U_{n}(\frac{t}{2}),$ where $U_{n}$ is the $n$th degree
Chebyshev polynomial of the second kind [11, p.15]. i.e.%
\begin{equation*}
U_{n}(x)=\frac{\sin ((n+1)\arccos x)}{\sin (\arccos x)},\ \ -1\leq x\leq 1.
\end{equation*}%
All the roots of the polynomial $U_{n}(x)$ are included in the interval $%
[-1,1]$ and can be found using the relation%
\begin{equation*}
x_{nk}=\cos \left( \frac{k\pi }{n+1}\right) ,\ \ k=\overline{1,n}.
\end{equation*}%
Therefore the characteristic polynomial in (6) can be written as%
\begin{equation*}
p_{n}(t)=(t^{2}-4)U_{n-2}(\tfrac{t}{2}).
\end{equation*}%
From [6, p. 2], the eigenvalues of the matrix $B$ are%
\begin{equation*}
t_{k}=2\cos \left( \frac{(k-1)\pi }{n-1}\right) ,\ k=\overline{1,n}.
\end{equation*}%
Then we get the eigenvalues of the matrix $A$ as%
\begin{equation*}
\lambda _{k}=a+2b\cos \left( \frac{(k-1)\pi }{n-1}\right) .
\end{equation*}%
Now we compute eigenvectors of the matrix $A$. All eigenvectors of the
matrix $A$ are the solutions of the following homogeneous linear equations
system%
\begin{equation}
(\lambda _{k}I_{n}-A)x=0  \tag{8}  \label{8}
\end{equation}%
where $\lambda _{k}$ is the $k$th eigenvalue of the matrix $A$ ($k=\overline{%
1,n}$). The equations system (8) is clearly written as%
\begin{equation}
\left. 
\begin{array}{r}
(\lambda _{k}-a)x_{1}-2bx_{2}=0 \\ 
-bx_{1}+(\lambda _{k}-a)x_{2}+bx_{3}=0 \\ 
bx_{2}+(\lambda _{k}-a)x_{3}+bx_{4}=0 \\ 
\vdots \ \ \ \  \\ 
bx_{n-2}+(\lambda _{k}-a)x_{n-1}-bx_{n}=0 \\ 
-2bx_{n-1}+(\lambda _{k}-a)x_{n}=0%
\end{array}%
\right\}  \tag{9}  \label{9}
\end{equation}%
Dividing all terms of the each equation in system (9) by $b\neq 0$,
substituting $m_{k}=\frac{\lambda _{k}-a}{2b},$ choosing $x_{1}=1$ then
solving the set of the system (9) we find the $j$th component of $k$th
eigenvector of the matrix $A$ as 
\begin{equation}
x_{jk}=\left\{ 
\begin{array}{l}
T_{j-1}(m_{k}),\ \ \ \ \ \ \ \ j=1,2,n-1,n \\ 
(-1)^{j}T_{j-1}(m_{k}),\ j=\overline{3,n-2}%
\end{array}%
\right. ;\ j,k=\overline{1,n}  \tag{10}  \label{10}
\end{equation}%
where $m_{k}=\frac{\lambda _{k}-a}{2b}$ and $T_{s}(.)$ is the $s-$th degree
Chebyshev polynomial of the first kind.
\end{proof}

\begin{theorem}
\bigskip Let $A^{\dagger }$ be an $n\times n$ tridiagonal matrix given by
(3). Then the eigenvalues and eigenvectors of $A^{\dagger }$ are 
\begin{equation*}
\lambda _{k}^{\dagger }=a-2b\cos \left( \frac{k\pi }{n}\right) ,\ k=%
\overline{1,n}
\end{equation*}%
and%
\begin{equation*}
y_{jk}^{\dagger }=\left\{ 
\begin{array}{l}
T_{\frac{2j-1}{2}}(m_{k}^{\dagger }),\ \ \ \ \ \ \ \ j=1,2,n-1,n \\ 
(-1)^{j}T_{\frac{2j-1}{2}}(m_{k}^{\dagger }),\ j=\overline{3,n-2}%
\end{array}%
\right. ;k=\overline{1,n}
\end{equation*}%
where $m_{k}^{\dagger }=\frac{\lambda _{k}^{\dagger }-a}{2b}$ and $T_{s}(.)$
is the $s-$th degree Chebyshev polynomial of the first kind [11, p. 14].
\end{theorem}

\begin{proof}
Let%
\begin{equation*}
S:=\left( 
\begin{array}{cccccc}
\frac{a}{b}+1 & 1 &  &  &  & 0 \\ 
1 & \frac{a}{b} & -1 &  &  &  \\ 
& -1 & \ddots & \ddots &  &  \\ 
&  & \ddots & \frac{a}{b} & -1 &  \\ 
&  &  & -1 & \frac{a}{b} & 1 \\ 
0 &  &  &  & 1 & \frac{a}{b}+1%
\end{array}%
\right) .
\end{equation*}%
From [7, Lemma 2, p. 65], we have the eigenvalues of the matrix $S$ as%
\begin{equation*}
\delta _{k}=\frac{a}{b}-2\cos \left( \frac{k\pi }{n}\right) ,\text{ for }k=%
\overline{1,n}
\end{equation*}%
Since the eigenvalues of $A^{\dagger }$ are $\lambda _{k}^{\dagger }=b\delta
_{k},$ the proof \ is completed.

The eigenvectors of $A^{\dagger }$ is the solution of the following linear
homogeneous equations system:%
\begin{equation}
(\lambda _{j}^{\dagger }I_{n}-A^{\dagger })y_{jk}=0  \tag{11}  \label{11}
\end{equation}%
where $\lambda _{j}^{\dagger }$ and $y_{jk}$\ are the $j$th eigenvalues and $%
k$th eigenvectors of the matrix $A^{\dagger }$ for $1\leq j,k\leq n.$Then
the solution of the equations system in (11) is%
\begin{equation*}
y_{jk}=\left\{ 
\begin{array}{l}
T_{\frac{2j-1}{2}}(m_{k}^{\dagger }),\ \ \ \ \ \ \ \ \ \ j=1,2,n-1,n \\ 
(-1)^{j}T_{\frac{2j-1}{2}}(m_{k}^{\dagger }),\ j=\overline{3,n-2}%
\end{array}%
\right. ;k=\overline{1,n}
\end{equation*}%
here $m_{k}^{\dagger }=\frac{\lambda _{k}^{\dagger }-a}{2b}$ and $T_{s}(.)$
is the $s-$th degree Chebyshev polynomial of the first kind.
\end{proof}

\section{\protect\bigskip The integer powers of the matrix A}

In this section we assume that $n$ is positive odd integer.

Since all the eigenvalues $\ \lambda _{k}\ (k=\overline{1,n})$ are simple,
each eigenvalue$\ \lambda _{k}$ corresponds single Jordan cells $J_{1}(\
\lambda _{k})$ in the matrix $J.$ Then, we write down the Jordan$^{\prime }$%
s forms of the matrix $A$

\begin{equation*}
J=diag(\lambda _{1},\lambda _{2},\lambda _{3},...,\lambda _{n}).
\end{equation*}

Using the equality $J=P^{-1}AP,$ we need the matrices $P$ and $P^{-1}$ and
derive the expressions for the $r$th power $(r\in \mathbb{Z})$ of the matrix 
$A.$Then%
\begin{equation*}
A^{r}=PJ^{r}P^{-1}.
\end{equation*}

From (10), we can write the eigenvectors matrix $P$\ as%
\begin{equation*}
P=[x_{jk}]=\left\{ 
\begin{array}{l}
T_{j-1}(m_{k}),\ \ \ \ \ \ \ \ \ j=1,2,n-1,n \\ 
(-1)^{j}T_{j-1}(m_{k}),\ j=\overline{3,n-2}%
\end{array}%
\right. k=\overline{1,n}
\end{equation*}%
where $T_{s}(.)$ is the $s-$th degree Chebyshev polynomial of the first kind.

First of all, let us obtain the inverse matrix $P^{-1}.$

Denoting $j$th column of the matrix $P^{-1}$ by $p_{j}$ then we have 
\begin{equation*}
p_{j}=\left( 
\begin{array}{c}
2T_{j-1}(m_{1}) \\ 
T_{j-1}(m_{2}) \\ 
T_{j-1}(m_{3}) \\ 
2T_{j-1}(m_{4}) \\ 
\vdots \\ 
2T_{j-1}(m_{n})%
\end{array}%
\right) ,\ j=1,n
\end{equation*}%
and%
\begin{equation*}
p_{j}=\left( 
\begin{array}{c}
(-1)^{j}4T_{j-1}(m_{1}) \\ 
(-1)^{j}2T_{j-1}(m_{2}) \\ 
(-1)^{j}2T_{j-1}(m_{3}) \\ 
(-1)^{j}4T_{j-1}(m_{4}) \\ 
\vdots \\ 
(-1)^{j}4T_{j-1}(m_{n})%
\end{array}%
\right) ,\ j=\overline{2,n-1}.
\end{equation*}

Hence we obtain%
\begin{equation*}
P^{-1}=\frac{1}{2n-2}(p_{1},p_{2},\ldots ,p_{n}).
\end{equation*}%
Let%
\begin{equation*}
A^{r}=PJ^{r}P^{-1}=U(r)=(u_{ij}(r)).
\end{equation*}%
Then%
\begin{equation*}
PJ^{r}=\left\{ 
\begin{array}{l}
\lambda _{k}^{r}T_{j-1}(m_{k}),\ \ \ \ \ \ \ \ \ j=1,2,n-1,n \\ 
(-1)^{j}\lambda _{k}^{r}T_{j-1}(m_{k}),\ j=\overline{3,n-2}%
\end{array}%
\right. k=\overline{1,n}.
\end{equation*}%
Hence

\begin{eqnarray*}
u_{ij}(r) &=&\frac{1}{2n-2}\left( \lambda
_{2}^{r}T_{i-1}(m_{2})T_{j-1}(m_{2})+\lambda
_{3}^{r}T_{i-1}(m_{3})T_{j-1}(m_{3})\right. \\
&&+2\dsum\limits_{\underset{k\neq 2,3}{k=1}}^{n}\lambda
_{k}^{r}T_{i-1}(m_{k})T_{j-1}(m_{k}))
\end{eqnarray*}%
where $i=\overline{1,n};\ j=1,n$ and%
\begin{eqnarray*}
u_{ij}(r) &=&\frac{1}{n-1}\left( 
\begin{array}{c}
\\ 
(-1)^{j}(\lambda _{2}^{r}T_{i-1}(m_{2})T_{j-1}(m_{2})+\lambda
_{3}^{r}T_{i-1}(m_{3})T_{j-1}(m_{3})) \\ 
\end{array}%
\right. \\
&&\left. +(-1)^{j}2\dsum\limits_{\underset{k\neq 2,3}{k=1}}^{n}\lambda
_{k}^{r}T_{i-1}(m_{k})T_{j-1}(m_{k})\right)
\end{eqnarray*}%
here $i=\overline{1,n};\ j=\overline{2,n-1}.$

\section{\textbf{Numerical \ examples}}

We can find the arbitrary integer powers of the $n$th order of the matrix,
where $n$ is positive odd integer.

\begin{example}
Let $n=3,r=3,a=1$ and $b=3.$\ Then we have%
\begin{equation*}
J=diag(\lambda _{1},\lambda _{2},\lambda
_{3})=diag(a,a+2b,a-2b)=diag(1,7,-5).
\end{equation*}%
Therefore%
\begin{equation*}
A^{3}=(q_{ij}(r))=(q_{ij}(3))=\frac{1}{4}\left( 
\begin{array}{ccc}
55 & 234 & 54 \\ 
117 & 109 & 117 \\ 
54 & 234 & 55%
\end{array}%
\right) .
\end{equation*}
\end{example}

\begin{example}
If $n=5,r=4,a=1$ and $b=3,$ then 
\begin{eqnarray*}
J &=&diag(\lambda _{1},\lambda _{2},\lambda _{3},\lambda _{4},\lambda _{5})
\\
&=&diag(1,7,-5,1+3\sqrt{2},1-3\sqrt{2}).
\end{eqnarray*}%
Therefore%
\begin{equation*}
A^{4}=q_{ij}(4)=\frac{1}{8}\left( 
\begin{array}{rrrrr}
595 & 672 & -756 & 216 & 162 \\ 
336 & 973 & -444 & 540 & 108 \\ 
-378 & -444 & 757 & -444 & -378 \\ 
108 & 540 & -444 & 973 & 336 \\ 
162 & 216 & -756 & 672 & 595%
\end{array}%
\right) .
\end{equation*}
\end{example}

\section{Complex Factorizations}

The well-known $F(x)=\{F_{n}(x)\}_{n=1}^{\infty }$ Fibonacci polynomials are
defined by $F_{n}(x)=xF_{n-1}(x)+F_{n-2}(x)$ with initial conditions $%
F_{0}(x)=0$ and $F_{1}(x)=1.$ For example if $x=1$ and $x=2,$\ then we obtain%
\begin{equation*}
F_{n}(1)=\{0,1,1,2,3,5,8,\ldots \}
\end{equation*}%
Fibonacci numbers and%
\begin{equation*}
F_{n}(2)=\{0,1,2,5,12,29,\ldots \}
\end{equation*}%
Pell numbers, respectively.

\begin{theorem}
Let the matrix $A$ be as in (2) with $a:=x$ and $b:=\mathbf{i}$ where $%
\mathbf{i}=\sqrt{-1}.$ Then%
\begin{equation*}
\det (A)=(x^{2}+4)F_{n-1}(x).
\end{equation*}
\end{theorem}

\begin{proof}
Applying Laplace expansion according to the first two and the last two rows
of the matrix $A$, we have%
\begin{equation*}
\det (A)=x^{2}D_{n-2}+4xD_{n-3}+4D_{n-4}
\end{equation*}%
here $D_{n}=\det (tridiag_{n}(-\mathbf{i},x,-\mathbf{i})).$ Since%
\begin{equation*}
\det (tridiag_{n}(-\mathbf{i},x,-\mathbf{i}))=F_{n+1}(x),
\end{equation*}%
we obtain%
\begin{eqnarray*}
\det (A) &=&x^{2}F_{n-1}(x)+4xF_{n-2}(x)+4F_{n-3}(x) \\
&=&x^{2}(xF_{n-2}(x)+F_{n-3}(x))+4xF_{n-2}(x)+4F_{n-3}(x) \\
&=&(x^{2}+4)(xF_{n-2}(x)+F_{n-3}(x))=(x^{2}+4)F_{n-1}(x).
\end{eqnarray*}%
Thus, the proof is completed.
\end{proof}

\begin{corollary}
Let the matrix $A$ be as in (2) with $a:=x$ and $b:=\mathbf{i.}$ Then the
complex factorization of generalized Fibonacci-Pell numbers is the following
form:%
\begin{equation*}
F_{n-1}(x)=\frac{1}{x^{2}+4}\dprod\limits_{k=1}^{n}\left( x+2\mathbf{i}\cos
\left( \frac{(k-1)\pi }{n-1}\right) \right)
\end{equation*}
\end{corollary}

\begin{proof}
Since the eigenvalues of the matrix $A$ from (4)%
\begin{equation*}
\lambda _{j}=x+2\mathbf{i}\cos \left( \frac{(j-1)\pi }{n-1}\right) ,\ j=%
\overline{1,n},
\end{equation*}%
the determinant of the matrix $A$ can be obtained as%
\begin{equation}
\det (A)=\dprod\limits_{k=1}^{n}\left( x+2\mathbf{i}\cos \left( \frac{%
(k-1)\pi }{n-1}\right) \right) .  \tag{12}  \label{12}
\end{equation}%
By considering (12) and Theorem 5, the complex factorization of generalized
Fibonacci-Pell numbers is obtained.
\end{proof}

\begin{theorem}
Let the matrix $A^{\dagger }$ be as in (3). $a:=1$ and $b:=\mathbf{i}$ where 
$\mathbf{i}=\sqrt{-1}.$ Then%
\begin{equation*}
\det (A^{\dagger })=\left\{ 
\begin{array}{c}
(1+2\mathbf{i})F_{n},\text{ if }a=1\text{ and }b=\mathbf{i} \\ 
(2+2\mathbf{i})P_{n},\text{ if }a=2\text{ and }b=\mathbf{i}%
\end{array}%
\right.
\end{equation*}%
where $F_{n}$ and $P_{n}$ are nth Fibonacci and Pell numbers, respectively.
\end{theorem}

\begin{proof}
Applying Laplace expansion according to the first two and last two rows the
determinant of the matrix $A^{\dagger }$, we have%
\begin{eqnarray}
\det (A^{\dagger }) &=&(a+b)^{2}\det (tridiag_{n-2}(-b,a,-b))  \TCItag{13}
\label{13} \\
&&-2b^{2}(a+b)\det (tridiag_{n-3}(-b,a,-b))  \notag \\
&&+b^{4}\det (tridiag_{n-4}(-b,a,-b)).  \notag
\end{eqnarray}%
If we get $a=1$ and $b=\mathbf{i}$ in (13), then we obtain%
\begin{eqnarray*}
\det (A^{\dagger }) &=&(1+\mathbf{i})^{2}\det (tridiag_{n-2}(-\mathbf{i},1,-%
\mathbf{i})) \\
&&+2(1+\mathbf{i})\det (tridiag_{n-3}(-\mathbf{i},1,-\mathbf{i})) \\
&&+\det (tridiag_{n-4}(-\mathbf{i},1,-\mathbf{i})).
\end{eqnarray*}%
Since%
\begin{equation*}
\det (tridiag_{n}(\mathbf{i},1,\mathbf{i}))=\det (tridiag_{n}(-\mathbf{i},1,-%
\mathbf{i}))
\end{equation*}%
from (1), we write%
\begin{eqnarray*}
\det (A^{\dagger }) &=&(1+\mathbf{i})^{2}F_{n-1}+2(1+\mathbf{i}%
)F_{n-2}+F_{n-3} \\
&=&(1+2\mathbf{i})F_{n}.
\end{eqnarray*}%
Similarly, we can easily obtain Pell numbers.
\end{proof}

\begin{corollary}
Let the matrix $A^{\dagger }$ be as in (3) with $a:=2$ and $b:=\mathbf{i}$.
Then the complex factorization of Fibonacci and Pell numbers are%
\begin{equation*}
F_{n}=\dprod\limits_{k=1}^{n-1}\left( 1-2\mathbf{i}\cos \left( \frac{k\pi }{n%
}\right) \right)
\end{equation*}%
and%
\begin{equation*}
P_{n}=\dprod\limits_{k=1}^{n-1}\left( 2-2\mathbf{i}\cos \left( \frac{k\pi }{n%
}\right) \right) .
\end{equation*}
\end{corollary}

\begin{proof}
Since the eigenvalues of the matrix $A^{\dagger }$ are%
\begin{equation*}
\lambda _{k}=a-2b\cos \cos \left( \frac{k\pi }{n}\right) ,\ k=\overline{1,n}
\end{equation*}%
and the determinant of the matrix $A^{\dagger }$ is multiplication of its
eigenvalues, we have%
\begin{eqnarray*}
F_{n} &=&\frac{1}{1+2\mathbf{i}}\dprod\limits_{k=1}^{n}\left( 1-2\mathbf{i}%
\cos \left( \frac{k\pi }{n}\right) \right) \\
&=&\dprod\limits_{k=1}^{n-1}\left( 1-2\mathbf{i}\cos \left( \frac{k\pi }{n}%
\right) \right)
\end{eqnarray*}%
and%
\begin{eqnarray*}
P_{n} &=&\frac{1}{2+2\mathbf{i}}\dprod\limits_{k=1}^{n}\left( 2-2\mathbf{i}%
\cos \left( \frac{k\pi }{n}\right) \right) \\
&=&\dprod\limits_{k=1}^{n-1}\left( 2-2\mathbf{i}\cos \left( \frac{k\pi }{n}%
\right) \right) .
\end{eqnarray*}%
Thus, the proof is comleted.
\end{proof}

\end{document}